\documentclass[eqnumis, pdf]{yjb}
\usepackage{amsfonts}
\usepackage{amsmath}
\usepackage{footnote}
\usepackage{multicol}
\usepackage{booktabs}
\usepackage[flushleft]{threeparttable}
\usepackage[font=small,labelfont=bf]{caption}
\usepackage{float}
\usepackage{hyperref}
\usepackage{esdiff}
\DeclareMathOperator{\sech}{sech}
\DeclareMathOperator{\arcsinh}{arcsinh}
\hypersetup{
colorlinks,
citecolor=red,
linkcolor=blue,
urlcolor=magenta}
\usepackage{mathrsfs,upref}

\begin{document}

%

\markboth{Y. J. Bagul, B. O. Fande}{Bounds for generalized circular and hyperbolic functions}

\title{Classical exponential bounds for $p$-generalized circular and hyperbolic functions}

\author{Yogesh J. Bagul $^{1}$\coraut, Bharti O. Fande$^{2,a}$}

\address{ $^{1}$Department of Mathematics\\ K. K. M. College, Manwath\\ Dist: Parbhani (M. S.) - 431505, India\\
$^{a}$ Research Scholar\\
$^{2}$Department of Mathematics\\ Government Vidarbha Institute of Science\\ and Humanities, Amravati (M. S.)-444604, India\\
	}
\emails{yjbagul@gmail.com, fandebharti@gmail.com}

\maketitle
\begin{abstract}
In this article, we obtain exponential bounds for the generalized circular and hyperbolic functions with one parameter $p$. Our results are natural generalizations of some existing results for classical circular and hyperbolic functions.
\end{abstract}
\subjclass{33B10, 33E30, 26D05.}  
\keywords{Generalized circular functions, generalized hyperbolic functions, eigenfunction,  exponential bounds, inequalities.}        


\vspace{10pt}

\section{Introduction}\label{sec1}
For $ 0 \leq x \leq 1, $ we know that
$$ \arcsin x = \int_0^x \frac{1}{(1-t^2)^{1/2}} \, dt $$
and $$ \frac{\pi}{2} = \arcsin 1 = \int_0^1 \frac{1}{(1-t^2)^{1/2}} \, dt. $$
Hence it is easy to define the function sine on $ [0, \pi/2] $ as the inverse of arcsine and it can be then extended on $ (-\infty, \infty).$
 On the other hand, an eigenfunction in the Dirichlet problem \cite{drabek} for the one-dimensional $p$-Laplacian occurs as the inverse of the function $ \phi_p: [0, 1] \rightarrow \mathbb{R} $ given by 
 $$ \phi_p(x) = \int_0^x \frac{1}{(1-t^p)^{1/p}} \, dt. $$
 This eigenfunction is associated with the eigenvalue containing $\pi_p = \frac{2\pi}{p \sin(\pi/p)}.$ Therefore, we can denote $ \phi_p(x) $ by $ \arcsin_p(x), $ i.e., 
 $$ \arcsin_p(x) = \int_0^x \frac{1}{(1-t^p)^{1/p}} \, dt $$
 and we evaluate 
 $$ \arcsin_p(1) = \frac{\pi}{p \sin(\pi/p)} = \frac{\pi_p}{2}. $$
 Thus, the inverse of $ \arcsin_p $ on $ [0, \pi_p/2] $ is a natural generalization of the sine function and it is denoted by $ \sin_p.$ For $ 1 < p < \infty, $ we call this function as the $p$-generalized sine function. The $p$-generalized sine function is clearly strictly increasing on $ [0, \pi_p/2]$ and $ \sin_p 0 = 0,$ $ \sin_p(\pi_p/2) = 1.$ Further, we can extend this function on $ (-\infty, \infty) $ by $2\pi_p$ periodicity. For instance, $ \sin_p $ is extended on $ [-\pi_p, \pi_p] $ by oddness and by defining $$ \sin_p(x) = \sin_p(\pi_p - x); \, \, x \in  [\pi_p/2, \pi_p]. $$
In the same way, other generalized circular and hyperbolic functions can be defined. These generalized functions are quite similar to the classical functions in many aspects and they have been studied extensively. For more details, we refer the reader to \cite{bushell, lang, yin, baricz, lindqvist, lindqvist1, lindqvist2, edmunds} and the references thereof. In this context, many researchers have obtained the inequalities involving generalized circular and hyperbolic functions. See, for example \cite{bhayo, bhayo1, huang, huang1, klen, neuman, wang, yin1} and the references therein. In this article, we establish exponential bounds for the generalized circular and hyperbolic functions. For the corresponding classical functions, this type of bounds can be seen in \cite{bagul, bagul1, bagul2, bagul3, chesneau, malesevic}. 

\section{Preliminaries and Lemmas}
In this section, we see the definitions of other generalized circular and hyperbolic functions, and lemmas for proving our main results. \\

Let $ 1 < p < \infty.$ Then the generalized cosine function $ \cos_p $ is defined by
$$ \cos_p x = \diff{\sin_p x}{x}, \, \, x \in [0, \pi_p/2]. $$
Hence, $$ \cos_p x = (1-\sin_p^p x)^{1/p}, \, \, x \in [0, \pi_p/2]. $$
From this we get $ \cos_p 0 = 1$ and $ \cos_p (\pi_p/2) = 0, $ and 
$$ \diff{\cos_p x}{x} = - \cos_p^{2-p} x \sin_p^{p-1} x, \, \, x \in [0, \pi_p/2].$$
It is to be noted that the generalized cosine function $ \cos_p $ is decreasing on $ [0, \pi_p/2].$ The generalized tangent function is defined by
$$ \tan_p x = \frac{\sin_p x}{\cos_p x}, \, \, x \in \mathbb{R} - \left\lbrace k \pi_p + \frac{\pi_p}{2}: k \in \mathbb{Z}\right\rbrace. $$
This implies 
$$ \diff{\tan_p x}{x} = 1 + \vert \tan_p x \vert^p, \, \, x \in (-\pi_p/2, \pi_p/2) $$
Or
$$ \diff{\tan_p x}{x} = 1 +  \tan_p^p x , \, \, x \in [0, \pi_p/2). $$
The generalized secant function is denoted by $ \sec_p $ and is defined as
$$ \sec_p x = \frac{1}{\cos_p x}, \, \, x \in [0, \pi_p/2). $$
Then it is easy to obtain 
$$ \sec_p^p x = 1 + \tan_p^p x, \, \, x \in (0, \pi_p/2) $$
and 
$$ \diff{\sec_p x}{x} = \sec_p x \tan_p^{p-1} x, \, \, x \in [0, \pi_p/2). $$

The generalized hyperbolic sine function $ \sinh_p $ is the inverse of the generalized inverse hyperbolic sine function $ \arcsinh_p  $ which is defined by 
$$ \arcsinh_p x = \left\{
           \begin{array}{cc}
             \int_0^x \frac{1}{(1+t^p)^{1/p}} dt, & \quad x \in [0, \infty), \\
             -\arcsinh_p (-x), & \quad x \in (-\infty, 0).
            \end{array}
            \right. $$
 The generalized hyperbolic cosine, tangent, and secant functions are denoted by $ \cosh_p$, $ \tanh_p, $ and $ \sech_p $ respectively. They are defined as
 $$ \cosh_p x = \diff{\sinh_p x}{x}, \, \, \tanh_p x = \frac{\sinh_p x}{\cosh_p x}, \, \, \sech_p x = \frac{1}{\cosh_p x}. $$  
 It is worth mentioning that all the above-generalized functions become classical ones for $ p = 2. $ 
 Furthermore, it is not difficult to obtain the following formulae:
 $$ \cosh_p^p x - \sinh_p^p = 1, \, \, x > 0, $$  
 $$ \diff{\cosh_p x}{x} = \cosh_p^{2-p} \sinh_p^{p-1} x, \, \, x \geq 0, $$   
 $$ \diff{\tanh_p x}{x} = 1 - \tanh_p^p x = \sech_p^p x, \, \, x \geq 0, $$ 
 and
 $$ \diff{\sech_p x}{x} = - \sech_p x \tanh_p^{p-1} x. $$  The proofs of our generalizations require the l'H\^{o}pital's rule of monotonicity \cite{anderson}. It is the key tool and is stated as
\begin{lemma}(\cite{anderson})\label{lemm2.1} 
Let $ f_1(x)$ and $f_2(x)$ be two real valued-functions which are continuous on $ [a, b] $ and derivable on $ (a, b) $, where $ -\infty < a < b < \infty $ and $ f
_2^{\prime}(x) \neq 0, $ for all $ x \in (a, b). $
Let, $$ A(x) = \frac{f_1(x) - f_1(a)}{f_2(x) - f_2(a)},\ x\in (a,b) $$
and $$ B(x) = \frac{f_1(x) - f_1(b)}{f_2(x) - f_2(b)},\ x\in (a,b). $$ Then, we have
\begin{itemize}
\item[(i)] $ A(x) $ and $ B(x) $ are increasing on $ (a, b) $ if $ f_1^{\prime}(x)/f_2^{\prime}(x) $ is increasing on $ (a, b) .$
\item[(ii)] $ A(x) $ and $ B(x) $ are decreasing on $ (a, b) $ if $ f_1^{\prime}(x)/f_2^{\prime}(x) $ is decreasing on $ (a, b). $
\end{itemize}
The strictness of the monotonicity of $ A(x) $ and $ B(x) $ depends on the strictness of the monotonicity of $ f_1^{\prime}(x)/f_2^{\prime}(x)$.  
\end{lemma}
Also, we prove the following auxiliary results.

\begin{lemma}\label{lemm2.2}
For $ p > 1, $ the function $ \xi(x) = \frac{x}{\tan_p x} $  is strictly decreasing in $ (0, \pi_p/2).$
\end{lemma}
\begin{proof}
After differentiation, we get
$$ (\tan_p^2 x ) \xi^{\prime}(x) = \tan_p x - x - x \tan_p^p x := \gamma(x) $$
and 
\begin{align*}
\gamma^{\prime}(x) &= 1 + \tan_p^p x -1 -\tan_p^p x -px \tan_p^{p-1} x \\
&= -p x \tan_p^{p-1} x < 0
\end{align*}
for $ p > 1$ and $x \in (0, \pi_p/2). $ This implies that $ \gamma(x)$ is strictly decreasing in  $ (0, \pi_p/2) $ and so $ \xi(x)$ is strictly decreasing in  $(0, \pi_p/2). $
\end{proof}

\begin{lemma}\label{lemm2.3}
If $ p > 1, $ then the function $ \varsigma(x) = \frac{x}{\tanh_p x} $  is strictly increasing for $ x \geq 0.$
\end{lemma}
\begin{proof}
Since$$ (\tanh_p^2 x) \varsigma^{\prime}(x)  = \tanh_p x - x \sech_p^p x := \delta(x)$$
and
$$ \delta^{\prime}(x) = p x \sech_p^{p-1} x \sech_p x \tanh_p^{p-1} x > 0$$
as $ p > 1 $ and $ x \geq 0.$ The statement follows by making the same argument as in the proof of Lemma \ref{lemm2.2}.
\end{proof}

\section{Main results}\label{sec2}

First, we establish classical exponential bounds for the generalized circular functions.

\begin{theorem}\label{thm3.1}
If $ p \geq 2 $ and $ x \in (0, \pi_p/2), $ then we have 
\begin{align}\label{eqn3.1}
e^{\alpha x^p} < \frac{\sin_p x}{x} < e^{\beta x^p}
\end{align}
with the best possible constants $ \alpha = \frac{2^p \ln(2/\pi_p)}{\pi_p^p} $ and $ \beta = -\frac{1}{p(p+1)}. $
\end{theorem}
\begin{proof}
Consider the function
$$ f(x) = \frac{\ln\left(\frac{\sin_p x}{x}\right)}{x^p} := \frac{f_1(x)}{f_2(x)}, $$
where $ f_1(x) = \ln\left(\frac{\sin_p x}{x}\right) $ and $ f_2(x) = x^p $ with $ f_1(0+) = 0 $ and $ f_2(0) = 0. $ Differentiation gives
$$ \frac{f_1^{\prime}(x)}{f_2^{\prime}(x)} = \frac{x \cos_p x - \sin_p x}{p \cdot x^p \sin_p x} := \frac{f_3(x)}{f_4(x)}, $$
where $ f_3(x) = x \cos_p x - \sin_p x $ and $ f_4(x) = p \cdot x^p \sin_p x $ with $ f_3(0) = 0 = f_4(0).$ Again differentiating with respect to $ x, $ we get
\begin{align*}
\frac{f_3^{\prime} (x)}{f_4^{\prime} (x)} &= \frac{-1}{p} \cdot \frac{x \cos_p^{2-p} x \sin_p^{p-1} x}{x^p \cos_p x + p \cdot x^{p-1} \sin_p x} \\
&= -\frac{1}{p} \cdot \frac{1}{\left(\frac{x}{\tan_p x}\right)^{p-1} + p \left(\frac{x}{\tan_p x}\right)^{p-2}},
\end{align*}
which is decreasing by Lemma \ref{lemm2.2} as $p \geq2.$ Applying Lemma \ref{lemm2.1} repeatedly we see that $ f(x) $ is strictly decreasing in $ (0, \pi_p/2). $ Consequently, we write 
$$ f\left(\frac{\pi_p}{2}-\right) < f(x) < f(0+). $$
Since $ f(0+) = -\frac{1}{p(p+1)} $ and $ f(\pi_p/2-) = \frac{2^p \ln(2/\pi_p)}{\pi_p^p}, $ we obtain (\ref{eqn3.1}).
\end{proof}

\begin{theorem}\label{thm3.2}
Let $ p \geq 1 $ and $ a \in (0, a), $ where $ a < \frac{\pi_p}{2}.$ Then the inequalities
\begin{align}\label{eqn3.2}
e^{\alpha_1 x^p} < \cos_p x < e^{\beta_1 x^p}
\end{align}
hold with the best possible constants $ \alpha_1 = \frac{\ln(\cos_p a)}{a^p} $ and $ \beta_1 = -\frac{1}{p}. $
\end{theorem}
\begin{proof}
Suppose $$ g(x) = \frac{\ln(\cos_p x)}{x^p} := \frac{g_1(x)}{g_2(x)}, \, \, x \in (0, a), $$
where $ g_1(x) = \ln(\cos_p x) $ and $ g_2(x) = x^p $ such that $ g_1(0) = 0 = g_2(0).$
Then 
\begin{align*}
\frac{g_1^{\prime}(x)}{g_2^{\prime}(x)} &= \frac{-\cos_p^{1-p} x \sin_p^{p-1} x}{p \cdot x^{p-1}} = -\frac{1}{p} \left(\frac{\tan_p x}{x} \right)^{p-1}
\end{align*}
which is strictly decreasing in $ (0, a) $ by Lemma \ref{lemm2.2}. By Lemma \ref{lemm2.1}, $ g(x) $ is also decreasing in $(0, a). $ As a consequence we have $$ g(a-) < g(x) < g(0+). $$
The limits $ g(0+) = -\frac{1}{p} $ and $ g(a-) = \frac{\ln(\cos_p a)}{a^p} $ give the desired inequalities (\ref{eqn3.2}).
\end{proof}

\begin{theorem}\label{thm3.3}
Let $ p \geq 2 $ and $ x \in (0, b)$ where $ b < \frac{\pi_p}{2}. $ Then the best possible constants $ \alpha_2 $ and $ \beta_2 $ such that the inequalities
\begin{align}\label{eqn3.3}
e^{\alpha_2 x^p} < \frac{x}{\tan_p x} < e^{\beta_2 x^p}
\end{align}
hold true are $ \frac{\ln(b/\tan_p b)}{b^p} $ and $ -\frac{1}{p+1} $ respectively.
\end{theorem}
\begin{proof}
Let $$ h(x) = \frac{\ln\left(\frac{x}{\tan_p x}\right)}{x^p} := \frac{h_1(x)}{h_2(x)}, \, \, x \in (0, b) $$
where $ h_1(x) = \ln\left(\frac{x}{\tan_p x}\right) $ and $ h_2(x) = x^p $ satisfying $ h_1(0+) = 0 = h_2(0). $ After differentiating we get
$$ \frac{h_1^{\prime}(x)}{h_2^{\prime}(x)} = \frac{1}{p} \cdot \frac{\tan_p x - x \sec_p^p x}{x^p \tan_p x} := \frac{h_3(x)}{h_4(x)}. $$
Here $ h_3(x) = \tan_p x - x \sec_p^p x $ and $ h_4(x) = x^p \tan_p x $ with $ h_3(0) = 0 = h_4(0).$ Differentiation and elementary computations yield
\begin{align*}
\frac{h_3^{\prime}(x)}{h_4^{\prime}(x)} &= \frac{-x \sec_p^p x \tan_p^{p-1} x}{x^p \sec_p^p x + p x^{p-1} \tan_p x} \\
&= \frac{-1}{\left(\frac{x}{\tan_p x}\right)^{p-1}+p\left(\frac{x}{\tan_p x}\right)^{p-2} \cos_p^p x}.
\end{align*}
Now $ \frac{x}{\tan_p x} $ is strictly decreasing in $ (0, b)$ by Lemma \ref{lemm2.2} and it is positive as well in $ (0, b).$ Moreover, $ \cos_p x $ is also positive decreasing in $ (0, b). $ So we conclude that $ h(x) $  is strictly decreasing in $ (0, b) $ by Lemma \ref{lemm2.1}. The inequalities (\ref{eqn3.3}) follow due to the limits $ h(0+) = \frac{-1}{p+1} $ and $h(b-) = \frac{\ln(b/\tan_p b)}{b^p}. $ This finishes the proof of Theorem \ref{thm3.3}.
\end{proof}

The particular cases of which the inequalities in the above theorems are generalizations are given below.
\begin{align}\label{eqn3.4}
e^{\frac{4 \ln(2/\pi)}{\pi^2} x^2} < \frac{\sin x}{x} < e^{- x^2/6}, \, \, x \in (0, \pi/2)
\end{align}
\begin{align}\label{eqn3.5}
e^{\frac{\ln(\cos a)}{a^2} x^2} < \cos x < e^{-x^2/2}, \, \, x \in (0, a), \, \, a < \frac{\pi}{2}
\end{align}
and
\begin{align}\label{eqn3.6}
e^{\frac{\ln(b/\tan b)}{b^2} x^2} < \frac{x}{\tan x} < e^{ -x^2/3}, \, \, x \in (0, b), \, \, b < \frac{\pi}{2}.
\end{align}
The inequalities (\ref{eqn3.4}), (\ref{eqn3.5}), and (\ref{eqn3.6}) have been obtained by putting $ p = 2 $ in Theorems \ref{thm3.1}, \ref{thm3.2}, and \ref{thm3.3}. These inequalities have already appeared in \cite{bagul, bagul1, chesneau}.

Next, we establish classical exponential bounds for the generalized hyperbolic functions.

\begin{theorem}\label{thm3.4}
For $ x \in (0, c), $ where $ c > 0 $ and $ p \geq 2, $ we have 
\begin{align}\label{eqn3.7}
e^{\lambda x^p} < \frac{x}{\sinh_p x} < e^{\eta x^p}
\end{align}
with the best possible constants $ \lambda = -\frac{1}{p(p+1)} $ and $ \eta = \frac{\ln(c/\sinh_p c)}{c^p}.$
\end{theorem}
\begin{proof}
Let us set $$ F(x) = \frac{\ln\left(\frac{x}{\sinh_p x}\right)}{x^p} := \frac{F_1(x)}{F_2(x)}, $$
where $ F_1(x) = \ln\left(\frac{x}{\sinh_p x}\right) $ and $ F_2(x) = x^p $ with $ F_1(0+) = 0 = F_2(0). $ Differentiation with respect to $ x $ yields
$$ \frac{F_1^{\prime}(x)}{F_2^{\prime}(x)} = \frac{1}{p} \cdot \frac{\sinh_p x - x \cosh_p x}{x^p \sinh_p x} := \frac{1}{p} \cdot \frac{F_3(x)}{F_4(x)} $$
Here $ F_3(0) = 0 = F_4(0) $ which forces us to differentiate again and get
\begin{align*}
\frac{F_3^{\prime}(x)}{F_4^{\prime}(x)} &= \frac{-x \cosh_p^{2-p} x \sinh_p^{p-1} x}{x^p \cosh_p x + p \cdot x^{p-1} \sinh_p x} \\
&= \frac{-1}{\left(\frac{x}{\tanh_p x}\right)^{p-1} + p \left( \frac{x}{\tanh_p x}\right)^{p-2}}.
\end{align*}
Utilization of Lemma \ref{lemm2.3} implies that the function $ \frac{F_3^{\prime}(x)}{F_4^{\prime}(x)} $ is strictly increasing in $ (0, c). $ Then the repeated use of Lemma \ref{lemm2.1} shows that the function $ F(x) $ is also strictly increasing in $ (0, c).$ Lastly, $ F(0+) = -\frac{1}{p(p+1)} $ by l'H\^{o}pital's rule and $ F(c-) = \frac{\ln(c/\sinh_p c)}{c^p}.$
\end{proof}

\begin{theorem}\label{thm3.5}
The best possible constants $ \lambda_1 $ and $ \eta_1$ such that the double inequality
\begin{align}\label{eqn3.8}
e^{\lambda_1 x^p} <  \cosh_p x < e^{\eta_1 x^p}
\end{align}
holds for $ p \geq 1 $ and $ x \in (0, d), $ where $ d > 0, $ are $ \frac{\ln(\cosh_p d)}{d^p} $ and $ \frac{1}{p} $ respectively.
\end{theorem}
\begin{proof}
Consider the function
$$ G(x) = \frac{\ln(\cosh_p x)}{x^p} := \frac{G_1(x)}{G_2(x)}, $$
where $ G_1(x) = \ln(\cosh_p x) $ and $ G_2(x) = x^p $ with $ G_1(0) = 0 = G_2(0).$ Then by elementary computations we get
\begin{align*}
\frac{G_1^{\prime}(x)}{G_2^{\prime}(x)} = \frac{\cosh_p^{2-p} x \sinh_p^{p-1} x}{p \cdot x^{p-1} \cosh_p x} = \frac{1}{p} \cdot \left( \frac{\tanh_p x}{x}\right)^{p-1}
\end{align*}
which is strictly decreasing in $ (0, 1) $ by Lemma \ref{lemm2.3} as $ p \geq 1.$ Applying Lemma \ref{lemm2.1}, we conclude that $ G(x) $ is strictly decreasing in $ (0, d). $ Thus $$ G(d-) < G(x) < G(0+). $$
Theorem \ref{thm3.5} follows from the limits $ G(0+) = \frac{1}{p} $ and $ G(d-) = \frac{\ln(\cosh_p d)}{d^p}. $
\end{proof}
\begin{theorem}\label{thm3.6}
If $ p \geq 2 $ and $ x \in (0, m), $ where $ m > 0, $ then the best possible constants $ \lambda_2 $ and $ \eta_2 $ such that the inequalities
\begin{align}\label{eqn3.9}
e^{\lambda_2 x^p} < \frac{\tanh_p x}{x}  < e^{\eta_2 x^p}
\end{align}
hold true are $ -\frac{1}{p+1} $ and $ \frac{\ln((\tanh_p m)/m)}{m^p} $ respectively.
\end{theorem}
\begin{proof}
Let us introduce the function
$$ H(x) = \frac{\ln\left(\frac{\tanh_p x}{x}\right)}{x^p} := \frac{H_1(x)}{H_2(x)}, $$
where $ H_1(x) = \ln\left(\frac{\tanh_p x}{x}\right) $ and $ H_2(x) = x^p $ with $ H_1(0) = 0 = H_2(0). $ By differentiation we get
$$ \frac{H_1^{\prime}(x)}{H_2^{\prime}(x)} = \frac{1}{p} \cdot \frac{x \sech_p^p x - \tanh_p x}{x^p \tanh_p x} := \frac{H_3(x)}{H_4(x)}, $$
where $ H_3(x) = x \sech_p^p x - \tanh_p x $ and $ H_4(x) = x^p \tanh_p x $ satisfying $ H_3(0) = 0 = H_4(0).$ Then
\begin{align*}
\frac{H_3^{\prime}(x)}{H_4^{\prime}(x)} &= \frac{-p \cdot x \sech_p^p x \tanh_p^{p-1} x}{p \cdot x^{p-1} \tanh_p x + x^p \sech_p^p x} \\
&= \frac{-p}{p \cdot \cosh_p^p x \left(\frac{x}{\tanh_p x}\right)^{p-2} + \left( \frac{x}{\tanh_p x}\right)^{p-1}}.
\end{align*}
Now $ \frac{x}{\tanh_p x} $ is clearly positive in $ (0, m) $ and it is also increasing by Lemma \ref{lemm2.3}. This implies that $ \frac{H_3^{\prime}(x)}{H_4^{\prime}(x)} $ is strictly increasing, fortiori $ H(x) $ is strictly increasing in $ (0, m) $ by Lemma \ref{lemm2.1}. Therefore 
$$ H(0+) = -\frac{1}{p+1} < H(x) < \frac{\ln((\tanh_p m)/m)}{m^p} = H(m-). $$ This completes the proof of Theorem \ref{thm3.6}.
\end{proof}
An immediate consequence of Theorem \ref{thm3.3} and Theorem \ref{eqn3.6} is the following corollary.
\begin{corollary}\label{cor3.7}
Let $ p \geq 2 $ and $ b < \frac{\pi_p}{2}. $  Then it is true that
\begin{align}\label{eqn3.10}
\frac{x}{\tan_p x} < \frac{\tanh_p x}{x}, \, \, x \in (0, b).
\end{align}
\end{corollary}
Indeed, the inequality (\ref{eqn3.10}) is valid in $ (0, \pi_p/2). $

In the end, we remark that by putting $ p = 2 $ in Theorems \ref{thm3.4}, \ref{thm3.5}, and \ref{thm3.6}, we achieve the following existing inequalities.
\begin{align}\label{eqn3.11}
e^{- x^2/6} < \frac{x}{\sinh x} < e^{\frac{\ln(c/\sinh c)}{c^2} x^2}, \, \, x \in (0, c), \, \, c > 0,
\end{align}
\begin{align}\label{eqn3.12}
e^{\frac{\ln(\cosh d)}{d^2} x^2} <  \cosh x < e^{x^2/2}, \, \, x \in (0, d), \, \, d > 0,
\end{align}
and
\begin{align}\label{eqn3.13}
e^{-x^2/3} < \frac{\tanh x}{x}  < e^{\frac{\ln((\tanh m)/m)}{m^2} x^2}, \, \, x \in (0, m), \, \, m > 0.
\end{align}

The above inequalities (\ref{eqn3.11}), (\ref{eqn3.12}), and (\ref{eqn3.13}) were established and discussed in \cite{bagul, bagul2, bagul3, chesneau}.

\end{document}